\documentclass[a4paper,10pt]{article}

\usepackage{geometry}
\usepackage{amsmath}
\usepackage{amssymb}
\usepackage{amsthm}
\usepackage[latin1]{inputenc}
\usepackage{eurosym}
\usepackage[dvips]{graphics}
\usepackage{graphicx}
\usepackage{epsfig}

\usepackage[hypertex]{hyperref}

\usepackage{ifthen}

\newcommand{\argmin}{\operatorname{argmin}}

\newcommand{\Rr}{{\mathbb{R}}}

\newcommand{\rar}{\rightarrow}

\newtheorem{theorem}{Theorem}

\newtheorem{lemma}{Lemma}
\newtheorem{proposition}{Proposition}
\newtheorem{definition}{Definition}

\begin{document}

\title{Mean field limit of a continuous time finite state game}
\author{Diogo
  A. Gomes
\footnote{Departamento de Matem\'atica and CAMGSD, IST, Lisboa, Portugal. e-mail: dgomes@math.ist.utl.pt}
, Joana Mohr \footnote{Instituto de Matem\'atica, UFRGS, 91509-900 Porto Alegre, Brasil. e-mail: rafars@mat.ufrgs.br}
and Rafael Rig\~ao Souza \footnote{Instituto de Matem\'atica, UFRGS, 91509-900 Porto Alegre, Brasil. e-mail: joana.mohr@ufrgs.br}
}

\date{\today} %%

\maketitle

\begin{abstract}
Mean field games is a recent area of study introduced by Lions and Lasry in a series of seminal papers in 2006. Mean field games
model situations of competition between large number of rational agents that play non-cooperative
dynamic games under certain symmetry assumptions. A key step is to develop a mean field model, in
a similar way to what is done in statistical physics in order to construct a mathematically tractable model.
A main question that arises
in the study of such mean field problems is the rigorous justification of the mean field models
by a limiting procedure.

In this paper we consider
the mean field limit of two-state Markov decision problem as the number of players $N\to \infty$. First we establish the existence
and uniqueness of a symmetric partial information Markov perfect equilibrium. Then we derive a mean field model and characterize
its main properties.
This mean field limit is a system of coupled ordinary differential equations with initial-terminal data.
Our main result is the convergence as $N\to \infty$ of the $N$ player game to the mean field model and an estimate of the rate of convergence.
%the mean field limit of a game between $N+1$ players as $N\to \infty$.
\end{abstract}

\thanks{D.G. was partially supported by CAMGSD/IST through FCT Program POCTI -
FEDER and by grants PTDC/MAT/114397/2009,
UTAustin/MAT/0057/2008, PTDC/EEA-ACR/67020/2006, PTDC/MAT/69635/2006,
and PTDC/MAT/72840/2006, and by the bilateral agreement Brazil-Portugal (CAPES-FCT) 248/09}

\thanks{R.R.S was partially supported by the bilateral agreement Brazil-Portugal (CAPES-FCT) 248/09, and
CAPES, PROCAD, Projeto Universal CNPq 471473/2007-3.}

\thanks{J.M was partially supported by the bilateral agreement Brazil-Portugal (CAPES-FCT) 248/09.}

\section{Introduction}

Mean field games is a recent area of research started by Pierre Louis Lions and Jean Michel Lasry
\cite{ll1, ll2, ll3, ll4} which attempts to understand the limiting behavior of systems involving very large
numbers of rational agents which play dynamic games under partial information and symmetry assumptions.
Inspired by ideas in statistical physics, Lions and Lasry introduced
a class of models in which the individual player contribution is encoded in a mean field
that contains only statistical properties about the ensemble.
A key question is how to
derive such effective or mean field equations that drive the system as well as to show convergence
as the number of agents increases to infinity. The literature on mean field games and its applications
is growing fast, for a recent survey see \cite{llg2} and reference therein. Applications of mean field games
arise in the study of growth theory in economics \cite{llg1} or environmental policy \cite{lst}, for instance, and it is likely that in
the future they will play an important r\^ole in economics and population models. There
is also a growing interest in numerical methods for
these problems \cite{lst}, \cite{DY}. The authors \cite{GMS} have also considered the discrete time, finite state
problem.

%The mean field limits for certain Nash equilibria were discussed in  {\bf REFERENCES} .

In this paper we consider the mean field limit of games between a large number of players that
are allowed to switch between two states. We are particularly interested in understanding
the limit as the number of players increases to infinity.
We should stress the the fact that we are considering only two states plays no special r\^ole and
we could easily generalize our results to any finite number of states.

In his PhD thesis, \cite{GueantT}, O. Gu\'eant
considered a problem with two states, modeling the labor market.
In this work he considered a continuum of individuals and a  labor market consisting of 2
sectors. Each individual has to decide on which sector
he or she is going to work.
This model consists in a coupled
systems of ordinary differential equations of the type that will be derived in section \ref{mfmsec}.
Another possible application of our models concerns the adoption or change of
a technology or services. For instance, a single agent faced with different social networks
will have a incentive to move to the network with more potential contacts, however other effects play a role in this player decision, such as the level of services, trouble of changing network, loss of contacts and so on. Another similar example
concerns switching between cell phone companies.

We start in Section \ref{nplayer} to model the $N+1$ player problem as a Markov decision process. We assume that $N$ of
the players have a fixed Markov switching strategy $\beta$ and then look at a reference player which looks to minimize a
certain performance criterion by choosing a suitable switching strategy $\alpha(\beta)$. This is a well know
Markov decision problem. The key novelty in this section consists in showing the existence of a Nash
equilibrium such that $\alpha(\beta)=\beta$ and its characterization through a non-linear ordinary
differential equation. In fact, this is a continuous time, partial information, symmetric version of the Markov perfect equilibrium notion
that has been studied (mostly in discrete time or stationary setting) in \cite{mpe1, mpe3, mpe4, mpe12}, and references therein. In
\cite{mpe5, mpe6} symmetric Markov perfect equilibrium are also considered, and in the last paper the case with an infinite number of players
is studied. In \cite{mpe11} the passage from discrete time to continuous time is considered for $N$ players in a war of attrition problem.

In Section \ref{mfmsec} we derive a mean field model for the optimal switching policy of a reference player
given the fraction $\theta(t)$ of players in one of the states. This model turns out to be
a coupled system of ordinary differential equations, where one equation governs the evolution of
$\theta$, and is subjected to initial conditions, whereas the other equation models
the evolution of the value function and has terminal data. We call this problem the initial-terminal
value problem. Initial terminal value problems are in fact a general feature in many mean field game
problems, see for instance \cite{ll1, ll2, ll3}. Of course, existence and uniqueness of solutions
is not immediate from the general ODE theory but, adapting the methods of Lions and Lasry we were successful
in establishing both.

Our main result, theorem \ref{teoconv},  is discussed in Section \ref{convsec}
where we prove
the convergence as the number of players $N\to \infty$ to a mean field model.

%The structure of the paper is: in section \ref{nplayer} we study games between $N+1$ players. In section \ref{mfmsec}
%we derive the mean field model and study some of its properties. Finally, in section \ref{convsec} we
%prove the main result of the paper which is the convergence as the number of players increases to infinity to
%the mean-field model.

%\subsection{TO DO NOTES}

%\begin{itemize}
%\item Finish introduction and bibliography.
%\item Check the name of Dynkin's formula - is this the standard for finite state Markov processes?
%\item Proof of proposition 1. I think this can probably be improved
%\item Explain in the example that $g$ stands for switching costs - give the interpretation in terms of social network or cell phone network switching.
%\item Fica a faltar rever as secções 3.4%, a secção 4.1 e 4.2 provavelmente precisam de ser melhoradas. Em particular
%faltam as provas na secção 4.1.
%\item Faltam as final remarks, talvez comparar com a literatura.
%\item I have changed the example to include switching costs - this makes the example slighty more general with the same level of complexity. This must be updated in section 3.4.
%\item section 3.4 ainda nao esta com a funcao g. Mas vale a pena?

%\item uniformizar a posicao dos argumentos das diversas funcoes.  Verificar se nao ha teoremas sem hipoteses suficientemente explicitas.

%\end{itemize}

\section{The $N+1$ player game}
\label{nplayer}

In this section we consider symmetric games between $N+1$ players under a symmetric partial information pattern.
We start by discussing the framework of this problem, namely controlled Markov Dynamics, \S\ref{cmdyn}, admissible
controls \S\ref{acdf}, and the individual player problem \S\ref{ippv}. Then in \S \ref{hypot} we discuss the main assumptions on running and terminal cost
that allow us to use Hamilton-Jacobi ODE methods, in \S \ref{thjode} to solve the $N+1$ player problem. Maximum principle type estimates are considered in \S
\ref{mp0} which are then applied to establishing the existence of Nash equilibrium solutions, \S\ref{eqsols}. This section ends with an example \S \ref{expl}.

\subsection{Controlled Markov Dynamics}
\label{cmdyn}

We consider a dynamic game between $N+1$ players
that are allowed to switch between two states denoted by $0$ and $1$.
We suppose that all players are identical and so the game is symmetric with respect to permutation of the players.
To describe the game we will use a reference player, which could be chosen as any one of the players.
%Each player knows his own state ($0$ or $1$), as well as the number $n$ of remaining players that are in state $0$.
%No further information is available to any individual player.

If we fix any player as the reference player, we will suppose that he knows his own state at time $t$,  given by  $i(t)$, and also knows the number $n(t)$ of remaining players that are in state $0$. $i(t)$ and $n(t)$ are stochastic processes that we will describe in the following.
No further information is available to the reference player.
Because the game is symmetric, the identity of the reference player is not important, and all other players have access to the same kind of information, i.e., its own state and the fraction of other players in state $0$.

%More precisely,  given a reference player,
We suppose the process $(n(t),i(t))$ is a
continuous time Markov process: the reference player follows a controlled Markov process $i(t)$ with transition rates from state $i$ to the other state $1-i$ given by  $\beta=\beta(i,n,t)$.
 More precisely we have
$$
\mathbb{P}\Big(i(t+h)=1-i\|n(t)=n,i(t)=i\Big)= \beta(i,n,t).h+o(h)\,,
$$
where $\lim \frac{o(h)}h =0$ when $ h \rar 0$.
Because of the symmetry of the game, all other players follow their own Markov process controlled by the same transition rate function $\beta:\{0,1\}\times\{0,...,N\}\times[0,+\infty) \rar [0,+\infty)$.
Note that the rate function $\beta$ is a deterministic time-dependent function, which makes $(n(t),i(t))$ a
non-time homogeneous Markov process.
We will suppose that $\beta$ is bounded and continuous as a function of time. We will refer to any Markov control with rate function  which is bounded and continuous on time, as an {\it admissible control}.

%$n(t)$ will also be a Markov process
 The transition rates of the process $n(t)$ are given by
\begin{align}\label{gamma}
\gamma^+_\beta (i,n,t)&= (N-n) \beta (1, n+1- i, t)\,,\\ \notag
\gamma^-_\beta(i,n,t)&= n \beta (0, n-i, t)\,,
\end{align}
where $\gamma^+_\beta$ stands for the transition rate from $n$ to $n+1$, and $\gamma^-_\beta$
is the  transition rate from $n$ to  $n-1$.
Note that $n+1-i$ is the total number of players in state $0$, as seen by a player (distinct from the reference player) in state $1$ whereas
$n-i$ is the number of players in state $0$ as seen by a player (distinct from the reference player) in state $0$.

 %and $\beta (i, k, t)$ is the conditional expectation of $\beta$ given $n(t)=k$. \footnote{é mesmo a esperança condicional? Acho que sim...}
More precisely, we have
\begin{align*}
\mathbb{P}\Big(n(t+h)=n+1\|n(t)&=n,i(t)=i\Big)= \gamma^+_\beta(i,n,t).h+o(h)\,,\\
\mathbb{P}\Big(n(t+h)=n-1\|n(t)&=n,i(t)=i\Big)= \gamma^-_\beta(i,n,t).h+o(h)\,,
\end{align*}
where $\lim \frac{o(h)}h =0$ when $ h \rar 0$.

We assume further that the state transitions of the different players
are independent, conditioned on $i$ and $n$.
Note that no information is available to any player concerning the state of any other individual player. All each player knows is its position and the number of other players in state $0$, which mean, the fraction of other players in each one of the states $0$ and $1$.

\subsection{A control problem}

Let now $T>0$, and let $c:\{0,1\}\times [0,1]\times \Rr_0^+\to \Rr$
and $\psi:\{0,1\}\times [0,1]\to \Rr$ be two (non-negative) functions. We will discuss the
precise hypothesis on $c$ and $\psi$ in section \ref{hypot}. We suppose $c\left(i,\frac n N,\beta\right)$ represents a running cost incurred by the reference player when he is in state $i$, $n$ of the remaining $N$ players are in state $0$ and this player has a transition rate $\beta$ from $i$ to $1-i$. We also suppose $\psi\left(i,\frac{n} N\right)$
represents a terminal cost incurred by the reference player at the terminal time $T$, if he ends up at time $T$ in state
$i$ and at that time $n$ of the other players are in state $0$.

If $A_t(i, n)$ denotes the event $i(t) = i$ and $n(t) = n$, the expected total cost of the reference player, giving the control $\beta$
and conditioned
on the event $A_t(i,n)$, will  be
\[V^{\beta}(i,n,t) =
\mathbb{E}^{\beta}_{A_t(i, n)} \left[ \int_t^T c\left(i(s),\frac{n(s)}N,\beta(s)\right)ds + \psi\left(i(T),
\frac{n(T)}N\right) \right]\;.
\]

We could be interested in finding an admissible control $\beta$ that minimizes, for each $(i,n,t)$, the function $V$ defined above. This however would require a cooperative behavior between players and it would be an usual stochastic optimal control problem. Instead, we are interested in finding an admissible control $\beta$ that is a symmetric Nash equilibria
for the game which we will soon describe. %That is, we will fix a special player, which will be called player $N+1$, which will try to

%\subsection{Admissible controls and the Dynkin formula}
\subsection{The Dynkin formula}
\label{acdf}

%Let $(\Omega,\mathcal{F},\mathbb{P})$ be a probability space and $\{\mathcal{F}_s\}_{s \geq 0}$ a filtration on $\mathcal{F}$. We recall here that a process $ \alpha:\Omega\times [0,+\infty)\rar \Rr$  is progressively measurable if it is adapted to $\mathcal{F}$ and its restriction to $\Omega\times [0,s]$ is $\mathcal{F}_s \times \mathcal{B}_s$ - measurable, where $\mathcal{B}_s$ is the Borel $\sigma$-algebra in $[0,s]$. Note that Markov controls at the form $\alpha(s)=\alpha(i(s),n(s),s)$, where $\alpha:\{0,1\}\times\{0,...,N\}\times[0,+\infty) \rar [0,+\infty)$  is a deterministic function, are progressively measurable provided $\alpha$ is continuous in $s$, for instance.

%We will call  a bounded progressively measurable process $ \alpha:\Omega\times [0,+\infty)\rar [0,+\infty)$ an {\it admissible control}. Note that, if $\alpha(i,n,s)$ is a Markov control, continuous as a function of time, then it is an admissible control .

Given two admissible controls $\beta$ and $\alpha$, we can define a non-time homogeneous Markov process
$(n(t), i(t))$ where the transition rates for $n$ are given by \eqref{gamma}
and the transition rate for $i$ is given by $\alpha$ as
$$ \mathbb{P}\Big(i(t+h)=1-i\|n(t)=n,i(t)=i\Big)= \alpha(i,n,t).h+o(h)\,,$$
where $\lim \frac{o(h)}h =0$ when $ h \rar 0$.
The idea here is that, while other players use the control $\beta$, the reference player can choose another control $\alpha$.%, in a more general class of controls.

Furthermore, %if $\alpha$ is any admissible control , and $\beta$ is any bounded, continuous in time, Markovian  control,
  we have that,
for any function $\varphi:\{0,1\} \times \{0,1,2,...,N\} \times [0,+\infty) \rar \Rr$, smooth in the last variable, and any $s>t$,
%$$ \mathbb{E}^{\beta,\alpha}_{A_t(i, n)} \left[\varphi(i(s),n(s),s)-\varphi(i,n,t)\right]= $$
\begin{equation}\label{Dynkin}  \mathbb{E}^{\beta,\alpha}_{A_t(i, n)} \left[\varphi(i(s),n(s),s)-\varphi(i,n,t)\right]
= \mathbb{E}^{\beta,\alpha}_{A_t(i, n)} \left[ \int_t^s
\frac{d\varphi}{dt} (i,n,r) + A^{\beta,\alpha}\varphi(i,n,r) dr  \right]\,,
 \end{equation}
where $A_t(i, n)$ still denotes the event $i(t) = i$ and $n(t) = n$, and
\begin{align}
\label{generator} &A^{\beta,\alpha}\varphi(i,n,r) = \alpha(i,n,r)(\bar \varphi- \varphi)(i,n,r)+\\\notag &\quad +
 \gamma^+_\beta(i,n,r) (\varphi(i,n+1,r)-\varphi(i,n,r)) +
 \gamma^-_\beta(i,n,r) (\varphi(i,n-1,r)-\varphi(i,n,r)) \,,
 \end{align}
where $\gamma^+_\beta$ and $\gamma^-_\beta$ are defined by (\ref{gamma}),
and $\bar \varphi(i,n,t) =  \varphi(1-i,n,t)$.
%In  the expectation above, the stochastic process $i(s)$ and $n(s)$ are the controlled Markov process whose transition rated are determined by $\alpha$  and $\beta$.
%%%In other words, a progressively measurable control $\alpha$ is admissible if \eqref{Dynkin}
%%%%holds for all bounded, continuous in time, Markovian  control $\beta$.
%%%% when $\alpha$ controls $i(t)$ and $n(t)$ is controlled by any Markov control $\beta$.

We call $A^{\alpha, \beta}$ the generator of the %(controlled) Markov
process and \eqref{Dynkin} the Dynkin's formula in
analogy to the Dynkin's formula in stochastic calculus.

\medskip

\subsection{Individual player point of view - introducing the game}
\label{ippv}

%Now we will suppose that a special player, which we will call player $N+1$,  try  to follow a different control $\alpha$, searching for a smaller expected total cost.

%Now we will divide the population of players in two groups, which will use different controls. One of this groups has only one special player, which we call player $N+1$, and will play the role of the reference player. The other group has all other players, which we will call the first $N$ players. As a consequence of the symmetry, the identity of such players is not important.

Now we suppose the reference player decides unilaterally  to use a different control, trying to improve its value function.%, while other players insist on the control $\beta$.
%We will call the other players the first $N$ players.

We will suppose the other players continue to  follow the Markov Chain with transition rate
$\beta(i,n,t)$, bounded and  continuous on time.
Therefore $n(t)$, the number of such players that are in state $0$, is a process to which correspond transition rates $\gamma^+_\beta$ and $\gamma^-_\beta$ as in (\ref{gamma}).

The reference player  looks for an admissible control $\alpha$,  possibly different from $\beta$, %and which we will allow to be an admissible control, not necessarily Markov,
that minimizes
\[
u(i,n,t,\beta,\alpha)= \mathbb{E}^{\beta,\alpha}_{A_t(i, n)} \left[ \int_t^T c\left(i(s),\frac{n(s)}N,\alpha(s)\right)ds + \psi\left(i(T),
\frac{n(T)}N\right) \right]\;.
\]
%while other players still use the control $\beta$.
That is, reference player looks for the control $\alpha$ which is a solution to the minimization problem
$$
u(i,n,t;\beta)= \inf_{\alpha} u(i,n,t,\beta,\alpha),
$$
where the minimization is performed over the set of all admissible controls $\alpha$.
We will call the function $u(i,n,t;\beta)$ above
the value function for the reference player  associated to
the strategy $\beta$ of the remaining $N$ players.
The control $\alpha$ that attains the minimum above can be called the best response of any player to a control $\beta$. %\footnote{poderiamos deslocar para ca a definicao de Nash equilibrium}

%(or simply the value function associated to $\beta$).

%Note that player $N+1$ makes the role of the reference player, but also trying a different control,which could be called the best response of any player to a control $\beta$. \footnote{poderiamos deslocar para ca a definicao de Nash equilibrium}
%If there exist a control that minimizes the function $u(i,n,t;\beta)$, for all $i,n,t$, this control will be the best response

\medskip

\subsection{Assumptions on running and terminal cost}
\label{hypot}

We discuss now the hypothesis used in this paper concerning the running and terminal costs.
We suppose that both the running cost $c=c(i,\theta,\alpha):\{0,1\}\times [0,1]\times \Rr_0^+\to \Rr$
and the terminal cost $\psi=\psi(i,\theta):\{0,1\}\times [0,1]\to \Rr$ are non-negative functions, as mentioned in the previous section, and also that they are
Lipschitz continuous in $\theta$. Of course, our results
would still be valid without any change if $c$ and $\psi$ are simply bounded below, instead of being non-negative.

We assume that $c(i,\theta,\alpha)$ is uniformly convex on $\alpha \geq 0$ and superlinear. We assume further that $c$ is differentiable, and $c\,'(\theta, \alpha)   $ is Lipschitz in the variable $\theta$.
%\[\lim_{\alpha\to \infty} c(i,\theta,\alpha)= \infty,\]
%for any $\theta\in [0,1]$.

For $p\in \Rr$ we define
\[
h(p, \theta, i)=\min_{\alpha\geq 0} \left[c(i, \theta, \alpha)+\alpha p\right].
\]
Note that $h$ is an increasing concave function of $p$, Lipschitz in $\theta$, and, hence, bounded below by
\[
\min_{\theta\in[0,1], i\in\{0,1\}}h(0, \theta, i).
\]
Because of the uniform convexity the minimum is achieved at a single point, and the function
\[
\alpha^*(p, \theta, i)= \argmin_{\alpha\geq 0} \left[c(i, \theta, \alpha)+\alpha p\right].
\]
is well defined. Furthermore we have
\begin{proposition}
The function $\alpha^*$ is locally Lipschitz in $p$, uniformly in $\theta\in [0,1]$. Furthermore it is uniformly
Lipschitz in $\theta$.
\end{proposition}
\begin{proof}

We will use the following inequalities, which are consequence of the uniform convexity of $c$: for all $\theta,\alpha\,', \alpha, p$ and $p\,'$, we have
\begin{equation}\label{uniconv1}c(\theta,\alpha')+\alpha' p\,'\geq c(\theta,\alpha)+\alpha p\,'+(c\,'(\theta,\alpha)+p\,')(\alpha'-\alpha)+\gamma |\alpha'-\alpha|^2,
\end{equation} and because $\alpha^*(p,\theta)$ is a minimizer,
 \begin{equation}\label{uniconv2}(c\,'(\theta,\alpha^*(p,\theta))+p)(\alpha'-\alpha^*(p))\geq 0\,.
 \end{equation}

 We will first prove that $\alpha^*$ is uniformly Lipschitz in $p$ : for that, we
 suppose  that $\theta$ is fixed. By the definition of $\alpha^*$ and  equation \eqref{uniconv1} we have
$$ c(\alpha^*(p))+\alpha^*(p) p\,'\geq c(\alpha^*(p\,'))+\alpha^*(p\,') p\,' \geq$$
$$\geq c(\alpha^*(p))+\alpha^*(p) p\,'+(c\,'(\alpha^*(p))+p\,')(\alpha^*(p\,')-\alpha^*(p))+\gamma |\alpha^*(p\,')-\alpha^*(p)|^2,
$$ hence

$$ 0\geq (c\,'(\alpha^*(p))+p)(\alpha^*(p\,')-\alpha^*(p))+(p\,'-p)(\alpha^*(p\,')-\alpha^*(p))+\gamma |\alpha^*(p\,')-\alpha^*(p)|^2.
$$

Now using equation \eqref{uniconv2} we obtain

$$ 0\geq (p\,'-p)(\alpha^*(p\,')-\alpha^*(p))+\gamma |\alpha^*(p\,')-\alpha^*(p)|^2.
$$
Therefore
$$|p\,'-p\,|\,|\alpha^*(p\,')-\alpha^*(p)|\geq\gamma |\alpha^*(p\,')-\alpha^*(p)|^2\,,
$$
which implies
$$|\alpha^*(p\,')-\alpha^*(p)|\leq \frac{1}{\gamma}|p\,'-p\,|.$$
This shows that $\alpha^*$ is uniformly Lipschitz in $p$.

Now we prove that $\alpha^*$ is Lipschitz in $\theta$: for that, we suppose that $p$ is fixed.  Again by the definition of $\alpha^*$ and  by equation \eqref{uniconv1} we have
$$c(\theta\,',\alpha^*(\theta))+ \alpha^*(\theta)p\geq c(\theta\,',\alpha^*(\theta\,'))+ \alpha^*(\theta\,')p
$$
$$\geq  c(\theta\,',\alpha^*(\theta))+ \alpha^*(\theta)p+c\,'(\theta\,',\alpha^*(\theta)  )(\alpha^*(\theta\,')-\alpha^*(\theta))+\gamma|\alpha^*(\theta\,')-\alpha^*(\theta) |^2,
$$
and then
$$0\geq c\,'(\theta\,',\alpha^*(\theta)  )(\alpha^*(\theta\,')-\alpha^*(\theta))+\gamma|\alpha^*(\theta\,')-\alpha^*(\theta) |^2.$$ Using equation \eqref{uniconv2} we get
$$0\geq [c\,'(\theta\,',\alpha^*(\theta)  )-c\,'(\theta,\alpha^*(\theta))](\alpha^*(\theta\,')-\alpha^*(\theta))
+\gamma|\alpha^*(\theta\,')-\alpha^*(\theta) |^2. $$
As $c\,'(\theta,\alpha)$ is Lipschitz in the variable $\theta$ we have

$$0\geq -K|\theta\,'-\theta |\,|\alpha^*(\theta)-\alpha^*(\theta\,')|+\gamma|\alpha^*(\theta\,')-\alpha^*(\theta) |^2 .$$
Therefore
$$|\alpha^*(\theta)-\alpha^*(\theta\,')|\leq \frac{K}{\gamma}|\theta -\theta\,'|\,, $$
which implies that $\alpha^*$ is Lipschitz in $\theta$.
\end{proof}

In section \ref{uniqmfg} we will present and discuss monotonicity assumptions on $\psi$ and $h$, namely conditions
\eqref{psimon} and \eqref{monprop}, which will be necessary to prove uniqueness of solutions of the mean field model that will be presented in section \ref{mfmsec}.

%Finally, we suppose that the function $\psi(i, \theta)$ is Lipschitz in $\theta$.

\medskip

\subsection{The Hamilton-Jacobi ODE}\label{hjode}
\label{thjode}

%Let $\bar u(i,n,t)=u(1-i,n,t)$.
Fix a admissible control $\beta$. % for the first $N$ players.
Consider the system of ODE´s indexed by $i$ and $n$ given by
\begin{align*} - \frac{d\varphi}{dt} (i,n,t) =&
 \gamma^+_\beta (i,n,t) (\varphi(i,n+1,t)-\varphi(i,n,t)) +
 \gamma^-_\beta (i,n,t) (\varphi(i,n-1,t)-\varphi(i,n,t))\\
 &+ h\left(\bar \varphi( i,n,t)- \varphi(i,n,t),  \frac n N, i\right) \,,
 \end{align*}
where %the infimum now is performed over the set of non-negative real numbers $a$, and
$\bar \varphi_{\beta}(i,n,t)=\varphi_{\beta}(1-i,n,t)$, and
$\gamma^+_\beta$ and $\gamma^-_\beta$ are given by \eqref{gamma}.
%$$
%\gamma^+_\beta(i,n,t)= (N-n) \beta (1, n+1- i, t)\,,$$
%$$
%\gamma^-_\beta(i,n,t)= n \beta (0, n-i, t).$$
%where \footnote{We can´t forget to put hypothesis ensuring the uniqueness of argmin}
% $$\alpha_*(i,k,t)=\argmin_{\alpha \geq 0} \left[c(i,k,\alpha)+\alpha(\bar \varphi( i,k,t)- \varphi(i,k,t)) \right].$$
Since $\gamma^-_\beta(i,0,t)=0$ and $\gamma^+_\beta(i,N,t)=0$, the evaluation of $\varphi$ at $n+1$ and $n-1$ does not cause problems outside the range, resp. when $n=N$ or $n=0$).
By setting $\varphi_n(i,t)=\varphi(i, n, t)$ we write the previous ODE in compact notation:
\begin{equation}\label{HJ_prim}
- \frac{d\varphi_n}{dt}  =
 \gamma^+_\beta (\varphi_{n+1}-\varphi_{n}) +
 \gamma^-_\beta (\varphi_{n-1}-\varphi_{n}) +
 h\left(  \bar \varphi_n- \varphi_n,\frac n N,i\right)\,.
 \end{equation}
This system of ODE is called the Hamilton-Jacobi (HJ) ODE for player $N+1$ associated to the strategy  $\beta$ of the remaining $N$ players. We start by proving a
verification theorem, which is completely analogous to the optimal control verification theorem, see \cite{FS} for instance.
\begin{theorem}
Let $\varphi_{\beta}$ be a solution to \eqref{HJ_prim} satisfying the
terminal condition $\varphi_{\beta}(i,n,T)=\psi\left(i,\frac{n}N\right)$.
%Let $\bar \varphi_{\beta}(i,n,t)=\varphi_{\beta}(1-i,n,t)$.
Then
$$u(i,n,t;\beta)=\varphi_{\beta}(i,n,t)\,.$$
Also, the control %Markovian control
\begin{equation}\label{control}\bar \alpha(\beta)(i,n,t)\equiv \alpha^* \left(\bar \varphi_{\beta}(i, n, t) - \varphi_{\beta}(i, n, t), \frac n N, i\right),\end{equation}
is admissible %bounded and continuous on time
and satisfies
$$u(i,n,t;\beta)= u(i,n,t,\beta,\bar \alpha(\beta))\,.$$
\end{theorem}

\medskip
Thus a classical solution to the HJ equation associated to $\beta$ is the value function corresponding to $\beta$ and determines an optimal admissible control $\bar \alpha(\beta)$, for the reference player.

\begin{proof}
Let $\alpha$ be an admissible control.
By (\ref{Dynkin}) we have
$$
 \mathbb{E}^{\beta,\alpha}_{A_t(i, n)} \left[\varphi_{\beta}(i(T),n(T),T)\right]-\varphi_{\beta}(i,n,t)
=  \mathbb{E}^{\beta,\alpha}_{A_t(i, n)} \left[ \int_t^T
\frac{d\varphi_{\beta}}{dt} (i,n,r) + A^{\beta,\alpha}\varphi_{\beta}(i,n,r) dr  \right]\,,
 $$
where $A^{\beta,\alpha}$ is given by \eqref{generator}.
%$$ A^{\beta,\alpha}\varphi_{\beta}(i,n,r) = \alpha(r)(\bar \varphi_{\beta}- \varphi_{\beta})(i,n,r)  + $$
%$$ +
% \gamma^+(i,n,r) (\varphi_{\beta}(i,n+1,r)-\varphi_{\beta}(i,n,r)) +
% \gamma^-(i,n,r) (\varphi_{\beta}(i,n-1,r)-\varphi_{\beta}(i,n,r))\,, $$
%and $\gamma^+$ and $\gamma^-$ are defined by (\ref{gamma}).
Adding
 $$\mathbb{E}^{\beta,\alpha}_{A_t(i, n)} \left[ \int_t^T
c\left(i(r),\frac{n(r)}N,\alpha(r)\right) dr  \right] +\varphi_{\beta}(i,n,t)\,,
$$
to both sides of the previous identity,
where $\alpha(r)=\alpha(i(r),n(r),r)$, and using the definition of $A^{\beta,\alpha}\varphi_{\beta}(i,n,r)$,
we have
%$$
%u(i,n,t;\beta,\alpha) = \varphi_{\beta}(i,n,t) +
%\mathbb{E}^{\beta,\alpha}_{A_t(i, n)} \left[ \int_t^T
%\frac{d\varphi_{\beta}}{dt} (i,n,r) + A^{\beta}\varphi_{\beta}(i,n,r) + c(i(r),n(r),\alpha(r)) dr  \right]\,. $$
%
% U we get
%
\begin{align} \nonumber
%\label{eqgrande}
&u(i,n,t;\beta,\alpha)=\\ \notag
%\mathbb{E}^{\beta,\alpha}_{A_t(i, n)} \left[ \int_t^T c(i(r),n(r),\alpha(r)) dr  +\varphi_{\beta}(i(T),n(T),T)\right]\\ \notag
&=\varphi_{\beta}(i,n,t)  + \mathbb{E}^{\beta,\alpha}_{A_t(i, n)} \Bigg[ \int_t^T
\frac{d\varphi_{\beta}}{dt} (i,n,r)   +
 \gamma^+_\beta(i,n,r) (\varphi_{\beta}(i,n+1,r)-\varphi_{\beta}(i,n,r))\\ \notag
 &+
 \gamma^-_\beta(i,n,r) (\varphi_{\beta}(i,n-1,r)-\varphi_{\beta}(i,n,r))
+ c\left(i,\frac{n}N,\alpha\right)
+ \alpha(r)(\bar \varphi_{\beta}- \varphi_{\beta})(i,n,r)   dr  \Bigg].
\end{align}
The equation above is valid for all admissible controls $\alpha$.
Now we can define
$$
\bar \alpha(\beta)(i,n,r)=\alpha^*\left(\bar \varphi_\beta( i,n,r)- \varphi_\beta(i,n,r), \frac n N, i\right),
$$
% \argmin_{a \geq 0} \left[c(i,n,a)+a(\bar \varphi( i,n,r)- \varphi(i,n,r))\right]\,.
%$$
which is a bounded continuous Markov control %(since it depends only on $i,n$ and $r$),
and therefore admissible.
%Now $u(i,n,t)= \inf_{\beta} u(i,n,t;\beta)$ and
We have
\begin{align*}u(i,n,t;\beta) \leq &u(i,n,t,\beta,\alpha^*) = \varphi_{\beta}(i,n,t) \\
&+ \mathbb{E}^{\beta,\alpha^*}_{A_t(i, n)} \Bigg[ \int_t^T
\frac{d\varphi_{\beta}}{dt} (i,n,r)   +
 \gamma^+_\beta(n,r) (\varphi_{\beta}(i,n+1,r)-\varphi_{\beta}(i,n,r))\\
 &+
 \gamma^-_\beta(n,r) (\varphi_{\beta}(i,n-1,r)-\varphi_{\beta}(i,n,r))
+ h\left(\bar \varphi_{\beta}(i,n,r)- \varphi_{\beta}(i,n,r), \frac n N, i\right)dr\Bigg]\,.
\end{align*}
Now, we see that the integrand vanishes since $\varphi_\beta$ is a solution to HJ, and therefore we have $u(i,n,t;\beta) \leq \varphi_{\beta}(i,n,t)$.

Now we prove the other inequality:
\begin{align*}&u(i,n,t;\beta)= \inf_{\alpha} u(i,n,t,\beta,\alpha)= \varphi_{\beta}(i,n,t)\\
& + \inf_{\alpha} \mathbb{E}^{\beta,\alpha}_{A_t(i, n)} \Bigg[ \int_t^T
\frac{d\varphi_{\beta}}{dt} (i,n,r)   +
 \gamma^+_\beta(i,n,r) (\varphi_{\beta}(i,n+1,r)-\varphi_{\beta}(i,n,r))\\
 & +
 \gamma^-_\beta(i,n,r) (\varphi_{\beta}(i,n-1,r)-\varphi_{\beta}(i,n,r))
+ c\left(i,\frac{n}N,\alpha\right)
+ \alpha(r)(\bar \varphi_{\beta}- \varphi_{\beta})(i,n,r)   dr  \Bigg]\\ \geq&
\;\varphi_{\beta}(i,n,t)+ \mathbb{E}^{\beta}_{A_t(i, n)} \Bigg[ \int_t^T
\frac{d\varphi_{\beta}}{dt} (i,n,r)   +
 \gamma^+_\beta(i,n,r) (\varphi_{\beta}(i,n+1,r)-\varphi_{\beta}(i,n,r)) \\
 &+
 \gamma^-_\beta(i,n,r) (\varphi_{\beta}(i,n-1,r)-\varphi_{\beta}(i,n,r))
+ h\left(\bar \varphi_{\beta}(i,n,r)- \varphi_{\beta}(i,n,r), \frac n N, i\right)   dr  \Bigg]\\
=& \;\varphi_{\beta}(i,n,t)\,,
\end{align*}
where the last equation holds because the integrand vanishes since $\varphi$ is a solution to HJ.

\medskip

Thus we have proved that $u(i,n,t;\beta)=\varphi_{\beta}(i,n,t)$.
\end{proof}

\subsection{Maximum principle}
\label{mp0}

Here we prove that the solutions to the Hamilton-Jacobi equations are uniformly bounded
independently on the control $\beta$. We denote by
\[
\|u(t)\|_{\infty}=\max_{n,i} |u_n(i, t)|,
\]
and
\[
M=\max_{(i, \theta)\in\{0,1\}\times [0,1]} |h(0, \theta, i)|.
\]

\begin{proposition}\label{maxpri}
Let $u$ be a solution to \eqref{HJ_prim}.
For all $0\leq t \leq T$ we have
%\[ \|u_n(t, i)\|_{\infty}\leq %\|u(T)\|_{\infty}+\max_{(i, \theta)\in\{0,1\}\times %[0,1]} |h(0, \theta, i)| (T-t).\]
\[
 \|u(t)\|_{\infty}\leq \|u(T)\|_{\infty}+2M (T-t).
 \]

\end{proposition}
\begin{proof}
Let $u$ be a solution to \eqref{HJ_prim}. Let $\tilde u=u+\rho (T-t)$. Then
$$
- \frac{d\tilde u_n}{dt}  =\rho+
 \gamma^+_\beta (\tilde u_{n+1}-\tilde u_{n}) +
 \gamma^-_\beta (\tilde u_{n-1}-\tilde u_{n}) + h\left(\bar {\tilde u}_n-\tilde u_n, \frac n N, i\right)\,.
$$
Let $(i,n,t)$ be a minimum point of $\tilde u$ on $\{0,1\}\times\{0,1, \cdots, N\}\times [0,T]$.
We have $\tilde u_n(i,t)\leq \tilde u_{n-1}(i,t)$ and $u_n(i,t)\leq u_{n+1}(i,t)$. This implies
$\gamma^-_\beta (\tilde u_{n-1}-\tilde u_n)\geq 0$ and $\gamma^+_\beta (\tilde u_{n+1}-\tilde u_n)\geq 0$.
We also have
$\tilde u_n(i,t) \leq  \tilde u_n(1-i,t) = \bar {\tilde u}_n(i,t)$, which implies $(\bar {\tilde u}_n- \tilde u_n)(i,t)\geq 0$.
Hence
$$-\frac{d\tilde u_n}{dt}(i,t)  \geq h\left(\bar{\tilde u}_n-\tilde u_n, \frac n N, i\right)+\rho\geq h\left(0,\frac n N, i\right)+\rho \,,$$
because $h(p, \theta, i)$ is monotone increasing in $p$. Furthermore, if we take $M<\rho< 2M$ we get
\[
-\frac{d\tilde u_n}{dt}(i,t)>0.
\]
This shows that the minimum of $\tilde u$ is achieved at $T$ hence
\[
u_n(t, i) \geq -\|u(T)\|_{\infty}-2M (T-t).
\]

\bigskip

Similarly, let
$(i,n,t)$ be a maximum point of $\tilde u$ on $\{0,1\}\times\{0,1, \cdots, N\}\times [0,T]$.
We have $\tilde u_n(i,t)\geq \tilde u_{n-1}(i,t)$ and $u_n(i,t)\geq u_{n+1}(i,t)$, and this implies
$\gamma^-_\beta (\tilde u_{n-1}-\tilde u_n)\leq 0$ and $\gamma^+_\beta (\tilde u_{n+1}-\tilde u_n)\leq 0$.
We also have
$\tilde u_n(i,t) \geq  \tilde u_n(1-i,t) = \bar {\tilde u}_n(i,t)$, which implies $(\bar {\tilde u}_n- \tilde u_n)(i,t)\leq 0$.
Hence
$$-\frac{d\tilde u_n}{dt}(i,t)  \leq h\left(\bar {\tilde u}_n-{\tilde u}_n, \frac n N, i\right)+\rho\leq h\left(0,\frac n N, i\right)+\rho \,,$$
because $h(p, \theta, i)$ is monotone increasing in $p$. Furthermore, if we take $-2 M < \rho <-M$ we get
\[
-\frac{d\tilde u_n}{dt}(i,t)<0.
\]
This shows that the maximum of $\tilde u$ is achieved at $T$ hence
\[
u_n(t, i)\leq \|u(T)\|_{\infty}+2M (T-t).
\]

\end{proof}

\subsection{Equilibrium solutions}
\label{eqsols}
%Let $\beta=\beta(i,n,t)$ be the control used by the first $N$ players, suppose that there exists a classical solution %$\varphi_{\beta}$ of the HJ equation associated to $\beta$.
%By the preceding theorem, we know that the Markovian control
%$$\alpha(\beta)(i,n,t)\equiv \argmin_{a>0} c(i,n,a)+
%a (\bar \varphi_{\beta} - \varphi_{\beta})(i,n,t),$$
%is the optimal control for the player $N+1$.

We now consider the equilibrium situation in which the best response of any player to a control $\beta$ is $\beta$ itself.

\medskip

\begin{definition}
Let $\beta$ be an admissible control. This control $\beta$ is a Nash equilibrium if  $\bar\alpha(\beta)=\beta$.
\end{definition}

\begin{theorem}\label{nash_ex_uniq}
There exists a Nash equilibrium, i.e, an admissible Markov control $\beta_*$, which satisfies
$\bar\alpha(\beta_*)=\beta_*$. Moreover, the Nash equilibrium is unique.
\end{theorem}
\begin{proof}
It suffices to observe that, by \eqref{control}
\[
\beta_*(i,n,t)=\alpha^*\left(\bar \varphi_{\beta^*}-\varphi_{\beta^*}, \frac n N, i\right),
\]
and hence the Markov control can be obtained by solving the system of  nonlinear differential equations
\begin{equation}
\label{eqhj}
- \frac{du_n}{dt}  =
 \gamma^+_n (u_{n+1}-u_{n}) +
 \gamma^-_n (u_{n-1}-u_{n}) + h\left(\bar u_n-u_n, \frac n N, i\right)\,,
\end{equation}
with terminal condition $u(i,n,T)=\psi\left(i,\frac{n}N\right)$, where $\gamma^{\pm}_n$ are given by
\begin{align}\label{gammaeq}
\gamma^+_n (i,t)&= (N-n) \alpha^* \left(\bar u_{n+1-i}-u_{n+1-i}, \frac{n+1- i}{N},1\right)\\ \notag
\gamma^-_n(i,t)&= n \alpha^* \left(\bar u_{n-i}-u_{n-i}, \frac{n-i}{N}, 0\right)\,.
\end{align}
Note that \eqref{eqhj} is well posed because $u_n$ is bounded and the righthand side is Lipschitz. Hence
it follows the existence and uniqueness of a Nash equilibrium.
\end{proof}

For the record we give here some properties of $\gamma^{\pm}_n$:
\begin{equation}\nonumber
%\label{gammap1}
|\gamma^\pm_n|\leq CN,
\end{equation}
and
\begin{equation}\nonumber
|\gamma^\pm_{n+1}-\gamma^\pm_n|\leq C+CN \|u_{n+1}-u_n\|_{\infty}.
\end{equation}

\subsection{An example}
\label{expl}

\medskip
Let $f:\{0,1\} \times [0,1] \rar \Rr$  and $g:\{0,1\} \times [0,1] \rar \Rr$  be two continuous function.
We take $$c(i,\theta,\alpha)
= f(i,\theta)+\frac{\alpha^2}{2}-\alpha g(i, \theta)\,.$$

This example could model, for instance, the marketshare of cellular companies where there are only two competitors and $N$ individual costumers.  If the state of the player represents the company he uses, we can think of $g(i,\theta)$ as a bonus the company $i$ offers customers of company $1-i$ in case they decide to switch.
If there are no such  bonus, we set $g=0$.

Then
$$
h(p, \theta, i)=
\min_{\alpha \geq 0} \left[c(i,\theta,\alpha)+\alpha p \right]=f(i,\theta)-\frac{((g(i,\theta)-p)^+)^2}{2}\,,$$
and
$$\alpha^*(p, \theta, i)=\argmin_{\alpha \geq 0} \left[c(i,\theta,\alpha)+\alpha p) \right]=(g(i, \theta)-p)^+\,.$$
%where $\bar u(i,n,t) = u(\bar i,n,t)$,
%is the optimal control associated to $\beta$.
%To find the Nash equilibria, we have to solve the equation
%\begin{equation}\label{nash_example}
%(u_{\beta}-\bar u_{\beta})^+(i,n,t) = \beta(i,n,t)\,.
%\end{equation}
%Therefore, we should analyze the particular form that the HJ equation has in this example. To do that, note that
%Then the Hamilton-Jacobi equation becomes
%$$ - \frac{du}{dt} (i,n,t) =$$
%$$ = (N-n) (u-\bar u)^+(1,n,t) (u(i,n+1,t)-u(i,n,t)) +
% n (u-\bar u)^+(0,n,t) (u(i,n-1,t)-u(i,n,t)) +$$
% $$+
% f_i(n)-\frac{((u-\bar u)^+(i,n,t))^2}{2}$$

Therefore \eqref{eqhj} becomes
\begin{equation}\label{exist_Nash} - \frac{du}{dt}  =
f-\frac{((u-\bar u+g)^+)^2}{2}
 + (N-n)(u-\bar u+g)^+_{1,n+1-i} (u_{n+1}-u_n) +
 n(u-\bar u+g)^+_{0,n-i} (u_{n-1}-u_n).
\end{equation}

By the results of section \ref{mp0} we know that any solution to \eqref{exist_Nash} is bounded a-priori.
Hence, if $f$ and $g$ are Lipschitz, \eqref{exist_Nash} has a unique solution $u$. Therefore, there
exists a unique Nash equilibrium.
%by $(u-\bar u)(i,n,t)$.

\smallskip

%(Probably this argument can be generalized giving us a proof for theorem \ref{nash_ex_uniq}.)

\section{A mean field model}
\label{mfmsec}

This section is dedicated to a mean field model which, as we will see
in the next section, corresponds to the limit as the number of players $N+1\to \infty$.
We start in \S \ref{tcpmfm} by discussing the model and its derivation under the mean field
hypothesis. Then, in \S \ref{existmfg} we address existence of solutions. Uniqueness of solutions (under a  monotonicity hypothesis similar to the ones in \cite{ll1, ll2}) is established in \S \ref{uniqmfg}. Finally, in \S \ref{btte},
we continue the study of the model problem from \S \ref{expl}.

\subsection{The control problem in the mean field model and Nash equilibria}\label{tcpmfm}

If the number of players is very large, we expect their distribution between the two states to be a deterministic function of the time $t$, as it would happen if we could somehow apply the law of large numbers.
So, we suppose the fraction of players in state $0$ is given by a deterministic function $\theta(t)$.
If all players use the same Markovian control $\beta=\beta(i,t)$, which now only depends on $i$ and $t$, then $\theta$ is a solution to
\begin{equation}\label{eq_pi}
\frac{d\theta}{dt} = (1-\theta) \beta_1-\theta \beta_0\hspace{2cm}    \theta(0)=\bar \theta \,,
 \end{equation}
where $\beta_i$ denotes the function $t \rar \beta(i,t)$, and $0\leq \bar \theta\leq 1$ is given and represents the initial distribution.
We suppose here that $\beta_i$ are continuous and bounded, for $i=0$ and $i=1$, and call such controls  admissible controls.

We can now consider the optimization problem
from a single player point of view. As before, we fix an individual player as the reference player and assume he can choose any admissible
 control $\alpha$, while other players have a probability distribution among states determined by \eqref{eq_pi}.
%themselves according to the function $\theta$.
Let
$$
u(i,t,\alpha)= \mathbb{E}^{\alpha}_{i(t)=i} \left[ \int_t^T c(i(s),\theta(s),\alpha(i(s),s))ds + \psi(i(T),\theta(T)) \right]\;,
$$
where $i(t)$ is %the state of the individual player at time $t$,
a controlled Markov chain switching between state $0$ and $1$ with rate $\alpha$. %$\alpha(t)=\alpha(i(t),t)$.
%where $A_t(i, n)$ denotes the event $i(t) = i$ and $n(t) = n$.
We assume this player  looks for an admissible control $\alpha$ which solves
$$
u(i,t)= \inf_{\alpha} u(i,t,\alpha).
$$
%We will suppose the minimization is performed in the class of all admissible controls.
Note that the situation is now simpler than in the $N+1$-player game, because $\theta$ is deterministic and the only stochastic process is $i(t)$ whose switching rate is controlled by  $\alpha$. We call
$u(i,t)$ the value function associated to the mean field distribution $\theta$.

Consider the following HJ equation:
\begin{equation}\label{HJ_MFG}
 - \frac{du}{dt}  =  % \frac{du}{d\theta} \frac{d\theta}{dt}+
h(\bar u- u, \theta, i)\,.
 \end{equation}
As in the verification theorem of \S \ref{thjode}, any solution $u$ to the equation above, with the terminal condition $u(i,T)=\psi(i,\theta(T))$, is the value function associated to $\theta$.
Furthermore, the optimal control is $\alpha^*(\bar u- u, \theta, i)$.
%Note that $\alpha^*$ is the optimal control given that other players use the control $\beta$, and thus is a function of $\beta$. We can write $\alpha^*= \Gamma(\beta)$.
%%%\begin{equation}\label{optcontrMFG}
%%%\alpha_{\theta}(i,t)\equiv \argmin_{a\geq0} c(i,\theta(t),a)+
%%%a (\bar u^{\theta} - u^{\theta})(i,t)\,.
%%%\end{equation}

Under the symmetry hypothesis, all players must use the same control when the Nash equilibria is attained. In other words, Nash equilibria is the fixed point to the operator described above, i.e., the operator that uses the control  $\beta$ to calculate $\theta$ as a solution to \eqref{eq_pi},
and after that determines the control $\alpha^*(\bar u- u, \theta, i)$ where $u$ is the solution to the HJ equation \eqref{HJ_MFG} determined by $\theta$, making the control $\alpha^*(\bar u- u, \theta, i)$ the image of $\beta$ under this operator.

%This means that the optimal control given above as a function of the population control $\beta$ must be equal to  $\beta$ when $\beta$ is the Nash equilibrium.
This leads then to the following system
of ordinary differential equations
\begin{equation}\label{MFG_eq}
\begin{cases}
- \frac{du}{dt}  = h(\bar u- u, \theta, i)\\
\frac{d\theta}{dt}=(1-\theta) \alpha^*(u(0,t)-u(1,t), \theta, 1)-\theta \alpha^*(u(1,t)-u(0,t), \theta, 0),
\end{cases}
\end{equation}
with the boundary data
\begin{equation}\label{MFG_condin}
\begin{cases}
u(i,T)=\psi(i, \theta(T))\\
\theta(0)=\bar \theta\,.
\end{cases}
\end{equation}
%where $0\leq \bar \theta\leq 1$ is given.

Note that from the ODE point of view this problem is somewhat non-standard as some of the
variables have initial conditions whereas other variables have prescribed terminal data. We call
this the initial-terminal value problem.

\subsection{Existence of Nash Equilibria in the MFG}
\label{existmfg}

We now address the existence of solutions to \eqref{MFG_eq} satisfying the initial-terminal conditions \eqref{MFG_condin}.
The proof of existence will be based upon a fixed point argument, using the operator $\xi$ described in the following, which is the analogous of the operator acting on the controls described in the last section, but now acting on distributions.

\begin{proposition}
There exists a solution to \eqref{MFG_eq} satisfying the initial-terminal conditions \eqref{MFG_condin}.
\end{proposition}
\begin{proof}
We need to solve (\ref{MFG_eq}) and (\ref{MFG_condin}) which can be rewritten as
\begin{equation}\label{edist}\frac{d\theta }{dt} =  (1-\theta) \alpha_1 -\theta \alpha_0     \hspace{2cm}    \theta(0)=\bar\theta\end{equation}
\begin{equation}\label{ehj} - \frac{du}{dt}  =
h(\bar u- u,\theta,i) \hspace{1cm} u(i,T)=\psi(i,\theta(T)) \, \end{equation}
where
$$ \alpha = \alpha^*(\bar u- u,\theta,i).$$

Let $\mathcal{F}$ be the set of continuous functions defined on $[0,T]$ and taking values in $[0,1]$, with the $C^0$ norm.
Consider the function $\xi:\mathcal{F}\rar \mathcal{F}$ that is obtained in the following way:
given $\theta \in \mathcal{F}$, let $u^{\theta}$ be the solution of equation (\ref{ehj}).
Let $\beta^{\theta} = \alpha^*(\bar u^{\theta} -u^{\theta} , \theta, i)\,,$
and then let $\xi(\theta)$ be the solution to
$\frac{d\theta }{dt} =  (1-\theta) \beta^{\theta}_1 -\theta \beta^{\theta}_0$ and $\theta(0)=\bar \theta$.

From standard ODE theory we know $\xi$ is a continuous function from $\mathcal{F}$ to $\mathcal{F}$.
Moreover, as $\beta$ is bounded, $\xi(\theta)$ is Lipschitz, with Lipschitz  constant $\Lambda$ independent of $\theta$.

Now consider the set $\mathcal{C}$ of all Lipschitz continuous function in $\mathcal{F}$ with Lipschitz constant bounded
by $\Lambda$. This is a set of uniformly bounded and equicontinuous functions. Thus, by Arzela-Ascoli, it is a relatively compact set. It is also clear that it is a convex set.
Hence, by Brouwer fixed point theorem, $\xi$ has a fixed point in $\mathcal{C}$.
\end{proof}

\subsection{Uniqueness of Equilibria}
\label{uniqmfg}
To establish uniqueness we need to use the monotonicity method of \cite{ll1, ll2}.

We will suppose the following monotonicity hypothesis on $\psi$:
\begin{equation}\label{psimon}
(x-y)[\psi(0,x)-\psi(0,y)] + (y-x)[\psi(1,x)-\psi(1,y)]
\geq 0 \,,
\end{equation}
%where $\Gamma$ is a non-negative constant.
for any $x$ and $y$ in $[0,1]$.
This hypothesis holds, for instance, if we suppose that $\psi$ is differentiable on its second variable, and
\begin{equation}\nonumber
    \frac{d\psi}{d\theta}(0,\theta) - \frac{d\psi}{d\theta}(1,\theta) \geq 0\,,
\end{equation}
or if we suppose that $\psi(0,\theta)$ is non-decreasing as function of $\theta$ and  $\psi(1,\theta)$ is non-increasing as function of $\theta$,
%is increasing in $\theta$, while $\psi(1,\theta)$ is decreasing.
which could be interpreted as a penalization on  crowded states.

Now, from the concavity of $h$ in $p$ we have, for all $p,q,\theta$ and $i$
\begin{equation}
\label{concvt}
h(q, \theta,i)-h(p, \theta,i)-\alpha^*(p,\theta,i) (q-p)\leq 0,
\end{equation}
because $\alpha^*(p, \theta,i) \in \partial^+_p h(p, \theta,i)$.
We suppose the additional monotonicity property
\begin{align}
\label{monprop}
\theta \Big(h(q, \tilde \theta,0)-h(q, \theta,0)\Big) &+\tilde \theta \Big(h(p, \theta,0)-h(p, \tilde \theta,0)\Big)\\\notag
+
(1-\theta)  \Big(h(-q, \tilde \theta,1)-h(-q, \theta,1)\Big)&+(1-\tilde \theta) \Big(h(-p, \theta,1)-h(-p, \tilde \theta,1)\Big)\leq -\gamma |\theta-\tilde \theta|^2,
\end{align}
for all $p, q\in \Rr$, for some $\gamma >0$.
This property will hold, for instance,
if
\begin{equation}
\label{mpropf0}
h(p, \theta,i)=h_0(p)+f(i,\theta),
\end{equation}
with $f$ satisfying
\begin{equation}
\label{mpropf}
(\theta-\tilde \theta)(f(0,\tilde \theta)-f(0,\theta))
+(\tilde \theta -\theta) (f(1,\tilde \theta)-f(1,\theta))
\leq -\gamma |\theta- \tilde \theta|^2.
\end{equation}
%SUGIRO TIRAR DAQUI The previous condition \eqref{mpropf} will hold, for instance if there exist uniformly convex functions $F_1(\theta)$ and $F_2(\theta)$  such that
%\[ f(\theta, 0)=F_1'(\theta), \quad f(\theta, 1)=-F_2'(\theta).\] ATE AQUI E ADICIONAR:

Note that the example of section \ref{expl} easily fits the previous conditions \eqref{mpropf0} and \eqref{mpropf} provided we suppose  $g$ is a constant function  and  the functions
$\theta \mapsto f(0,\theta)$ and
$\theta \mapsto f(1,\theta)$
 %differentiable with derivative respectively greater than $\gamma /2$ and lesser than  $-\gamma /2$: this could be a consequence of the fact that the current cost is greater when the reference player is in the more crowded state (i.e when $\theta=1$ if $i=0$ and when $\theta=0$ if $i=1$).
satisfy
$$(\tilde \theta- \theta)(f(0,\tilde \theta)-f(0,\theta))
\geq \frac{\gamma}{2} |\theta- \tilde \theta|^2$$
and
$$(\tilde \theta -\theta) (f(1,\tilde \theta)-f(1,\theta))
\leq -\frac{\gamma}{2} |\theta- \tilde \theta|^2,$$
which could be seen as a consequence of the fact that the running cost is greater when the reference player is in the more crowded state (i.e. when $\theta=1$ if $i=0$ and when $\theta=0$ if $i=1$).

Then
%\footnote{talvez dizer porque esta hipotese eh razoavel - estou convencido, mas o leitor pode nao se convencer}
, using \eqref{concvt} and \eqref{monprop} we obtain
\begin{equation}
\label{keyineq}
\begin{split}
&\theta \Big(h(q, \tilde\theta,0)-h(p, \theta,0)-\alpha^*(p,\theta,0) (q-p)\Big) +\tilde \theta \Big(h(p, \theta,0)-h(q, \tilde \theta,0)-\alpha^*(q,\tilde \theta,0) (p-q)\Big)\\
&+(1-\theta)\Big(h(-q,\tilde \theta,1)-h(-p, \theta,1)-\alpha^*(-p,\theta,1)(p-q)   \Big)\\
&+(1-\tilde \theta)\Big(h(-p,\theta,1)-h(-q,\tilde \theta,1)-\alpha^*(-q,\tilde \theta,1)(q-p)  \Big)  \leq
-\gamma |\theta-\tilde \theta|^2.
\end{split}
\end{equation}

\begin{theorem}\label{ex_Nahs_MFG}
Under the monotonicity hypothesis \eqref{psimon} and \eqref{monprop},
the system (\ref{edist}) and (\ref{ehj}) has a unique solution $(\theta,u)$.
%Moreover, equation (\ref{optcontrMFG}) gives a Nash equilibrium to the limiting mean field model, and the Nash equilibria is %unique.
\end{theorem}
\begin{proof}
To establish uniqueness we will use monotonicity argument from \cite{ll1, ll2}.

Suppose $(\theta,u)$ and $(\tilde \theta, \tilde u)$ are solutions of (\ref{edist}) and (\ref{ehj}).
At the initial point $t=0$ we have that $(\theta -\tilde \theta)(u-\tilde u )=0$ and $((1-\theta) -(1- \tilde \theta))(\bar u-\bar{\tilde u} )=0$, where $u(t)=u(0,t)$ and $\bar u=u(1,t)$, and similarly for $\tilde u$.
%At the initial and terminal points,
%respectively $t=0$ and $t=T$ we have that $(\theta -\tilde \theta)(u-\tilde u )=0$ and $((1-\theta) -(1- \tilde \theta))(\bar u-\bar{\tilde u} )=0$, where $u(t)=u(0,t)$ and $\bar u=u(1,t)$, and similarly for $\tilde u$.
Then
$$
(\theta -\tilde \theta)(u-\tilde u)_t=(\theta -\tilde \theta) [-h(\bar u-u,\theta ,0) +h(\bar{\tilde u}-\tilde u,\tilde\theta, 0)],
$$
and
$$ ((1-\theta) -(1- \tilde \theta))(\bar u-\bar{\tilde u} )_t= ((1-\theta) -(1-\tilde \theta)) [-h(u-\bar u,\theta,1)+h(\tilde u-\bar{\tilde u},\tilde\theta,1)].
$$
Furthermore,
\begin{align*}
(u-\tilde u) ( \theta -\tilde \theta)_t=&(u-\tilde u)[(1-\theta) \alpha^*(u-\bar u, \theta, 1)-\theta \alpha^*(\bar u-u, \theta, 0)
\\
&- (1-\tilde \theta) \alpha^*(\tilde u-\bar{\tilde u}, \tilde\theta, 1)+\tilde\theta \alpha^*(\bar{\tilde u}-\tilde u, \tilde\theta, 0)],
\end{align*}
and
\begin{align*}
 (\bar u-\bar{\tilde u}) ((1- \theta) -1+\tilde \theta)_t=&(\bar u-\bar{\tilde u})[\theta \alpha^*(\bar u-u), \theta, 0)-(1-\theta) \alpha^*(u-\bar u, \theta, 1)\\
 &-\tilde\theta \alpha^*(\bar{\tilde u} -\tilde u, \tilde\theta, 0) +(1-\tilde \theta) \alpha^*( \tilde u-\bar{\tilde u}, \tilde\theta, 1)].
\end{align*}
Hence,
\begin{align*}
&\qquad \frac{d}{dt}\Big((\theta -\tilde \theta)(u-\tilde u )+ ((1-\theta) -(1- \tilde \theta))(\bar u-\bar{\tilde u} )    \Big)=\\
=&\theta\Big(-h( \bar u-u,\theta ,0) +h(\bar{\tilde u}-\tilde u,\tilde\theta, 0) + [( \bar u-\bar{\tilde u})-(u-\tilde u)]\alpha^*(\bar u-u, \theta, 0)        \Big)\\
&+\tilde \theta   \Big( h( \bar u-u,\theta ,0) -h(\bar{\tilde u}-\tilde u,\tilde\theta, 0) + [-( \bar u-\bar{\tilde u})+(u-\tilde u)]\alpha^*(\bar{\tilde u} -\tilde u, \tilde\theta, 0)        \Big)\\
&
+(1- \theta )  \Big(-h(u-\bar u,\theta,1)+h({\tilde u}-\bar{\tilde u},\tilde\theta,1)  +[(u-\tilde { u})- ( \bar u-\bar{\tilde u})  ] \alpha^*(u-\bar u, \theta, 1) \Big)\\
&+
(1- \tilde\theta )  \Big(h(u-\bar u,\theta,1)-h({\tilde u}-\bar{\tilde u},\tilde\theta,1)  +[-(u-\tilde u)+ ( \bar u-\bar{\tilde u})  ] \alpha^*( \tilde u-\bar{\tilde u}, \tilde\theta, 1)       \Big).
\end{align*}
Then, by using \eqref{keyineq}, with $p=\bar u-u$ and $q=\bar{\tilde u}-{\tilde u}$, we obtain
\begin{equation}
\frac{d}{dt}\Big((\theta -\tilde \theta)(u-\tilde u )+ ((1-\theta) -(1- \tilde \theta))(\bar u-\bar{\tilde u} )    \Big)
\leq - \gamma |\theta - \tilde \theta|^2.
\end{equation}
Integrating the previous equation between $0$ and $T$, and using the terminal conditions, we have that
% $u(i,T)=\psi(i,\theta(T))$, we have, for $i=0$, that
$$(\theta(T) -\tilde \theta(T))[\psi(0,\theta(T))- \psi(0,\tilde\theta(T)) ]+ (\tilde \theta(T)-\theta(T))  [\psi(1,\theta(T))-\tilde{\psi(1,\theta(T))} ]
\leq - \gamma \int _0^T|\theta(s) - \tilde \theta(s)|^2ds.$$
Hence by the monotonicity condition \eqref{psimon} we get
$$  0\leq -\gamma \int _0^T|\theta(s) - \tilde \theta(s)|^2ds,   $$
which implies that $\theta(s)=\tilde \theta(s)$ for all $s\in [0,T]$. Therefore,
we have the uniqueness for $\theta$. Then, once $\theta$
is known to be  unique, we obtain by a standard ODE argument that $u=\tilde u$.
\end{proof}

\subsection{Back to the example}\label{btte}

Just to illustrate, equations (\ref{MFG_eq}), in the special case of the example of section \ref{expl}, and supposing $g$ is a constant function,  becomes

\begin{equation}\nonumber%\label{pi_ex}
\frac{d\theta}{dt}= (1-\theta) (u-\bar u)^+_1 - \theta (u-\bar u)^+_0\,, \end{equation}
and %equation (\ref{HJ_MFG}) assumes the form
\begin{equation}\nonumber%\label{HJ_ex}
 - \frac{du}{dt}  =
 f(i,\theta)-\frac{((u-\bar u+g)^+)^2}{2}\,.
% f(\theta)-\frac{[(g(\theta)-(u-\bar u))^+]^2}{2}\,.
 \end{equation}
As we have already seen, provided the  condition \eqref{mpropf} holds and given the initial-terminal condition
$$    \theta(0)=\bar\theta,\,\,\,\, u(i,T)=\psi(i,\theta(T)) $$ the system above has a unique solution.

\section{Convergence}
\label{convsec}

This last section addresses the convergence as the number of players tends to infinity to the mean field model
derived in the previous section.

We start this section by discussing some preliminary estimates in \S\ref{pest}. Then, in
\S \ref{uest} we establish uniform estimates for $|u_{n+1}-u_n|$, which are essential
to prove our main result, theorem \ref{teoconv}, which is discussed in \S\ref{convsubsec}.
This theorem shows that the model derived in the previous section can be obtained as an appropriate limit
of the model with $N+1$ players discussed in section \ref{nplayer}.

\subsection{Preliminary results}
\label{pest}

Consider the system of ordinary differential equations
\begin{equation}
\label{unpert}
-\dot z_n=a_n(t) (z_{n+1}-z_n)+b_n(t) (z_{n-1}-z_n)+\mu_n(t)(\bar z_n-z_n),
\end{equation}
with $a_n(t), b_n(t), \mu_n(t)\geq 0$. Here $z_n=(z_n^0, z_n^1)$, $a_n=(a_n^0, a_n^1)$, etc.
We assume further that $a_N=0$ and $b_0=0$.
%, with the usual convention for
%operator.

We write \eqref{unpert} in compact form as
\begin{equation}
\label{unpert2}
-\dot z(t)=M(t)z(t).
\end{equation}
The solution to this equation with terminal data $z(T)$ can be written as
\begin{equation}\label{unpert3}
z(t)=K(t,T)z(T),
\end{equation}
where $K(t, T)$ is the fundamental solution to \eqref{unpert2} with $K(T,T)=I$. Note that equations \eqref{unpert2} and \eqref{unpert3} imply
\begin{equation}\label{unpert4}\frac{d}{dt}K(t,T)=-M(t)K(t,T). \end{equation}

\begin{lemma}
\label{contprop}
For $t<T$ we have
\[
\|z(t)\|_{\infty}\leq \|z(T)\|_{\infty}.
\]
Furthermore, if $z(T)\leq 0$ then $z(t)\leq 0$.
\end{lemma}
\begin{proof} Let $z$ be a solution of \eqref{unpert2}, and fix $\epsilon>0$. We define $\tilde z= z+\epsilon (t-T)$. Hence $\tilde z$ satisfies
$$
-\dot {\tilde z}_n=-\epsilon +a_n(t) (\tilde z_{n+1}-\tilde z_n)+b_n(t) (\tilde z_{n-1}-\tilde z_n)+\mu_n(t)(\bar {\tilde z}_n-\tilde z_n).
$$
Let $(i,n,t)$ be a maximum point of $\tilde z$ on $\{0,1\}\times\{0,1, \cdots, N\}\times [0,T]$.
We have $\tilde z_n(i,t)\geq \tilde z_{n-1}(i,t)$ and $z_n(i,t)\geq z_{n+1}(i,t)$, also  $\tilde z_n(i,t) \geq  \tilde z_n(1-i,t) = \bar {\tilde z}_n(i,t)$, this implies
$b_n(t) (\tilde z_{n-1}-\tilde z_n)\leq 0$ and $a_n(t) (\tilde z_{n+1}-\tilde z_n)\leq 0$ and $\mu_n(t)(\bar {\tilde z}_n- \tilde z_n)(i,t)\leq 0$.
Hence
$$-\frac{d\tilde z_n}{dt}(i,t)  \leq -\epsilon. $$
This shows that the maximum of $\tilde z$ is achieved at $T$. Therefore, for all $(j,m,t)$,
$$z_m(j,t)+\epsilon(t-T)=\tilde z_m(j,t)\leq \tilde z_n(i,T)=z_n(i,T)    $$
Letting $\epsilon\to 0$, we get
$$z_m(j,t)\leq \max_{n, i} z_n(i,T).
$$

From this equation we have the following conclusions:
\begin{enumerate}
\item
if $z(T)\leq 0$, we then have
$\;z_m(j,t)\leq 0\;$, for all $(j,m,t)$, and so $z(t)\leq 0$;
\item for all $(j,m,t)$,
$$z_m(j,t)\leq   \|z(T)\|_{\infty}.  $$
\end{enumerate}

Now we define $\tilde z= z+\epsilon (T-t)$. Hence $\tilde z$ satisfies
$$-\dot {\tilde z}_n=\epsilon +a_n(t) (\tilde z_{n+1}-\tilde z_n)+b_n(t) (\tilde z_{n-1}-\tilde z_n)+\mu_n(t)(\bar {\tilde z}_n-\tilde z_n).$$

Let $(i,n,t)$ be a minimum  point of $\tilde z$ on $\{0,1\}\times\{0,1, \cdots, N\}\times [0,T]$.
We have $\tilde z_n(i,t)\leq \tilde z_{n-1}(i,t)$ and $z_n(i,t)\leq z_{n+1}(i,t)$, also  $\tilde z_n(i,t) \leq  \tilde z_n(1-i,t) = \bar {\tilde z}_n(i,t)$. This implies
$b_n(t) (\tilde z_{n-1}-\tilde z_n)\geq 0$, and $a_n(t) (\tilde z_{n+1}-\tilde z_n)\geq 0$ and $\mu_n(t)(\bar {\tilde z}_n- \tilde z_n)(i,t)\geq 0$.
Therefore we have
$$-\frac{d\tilde z_n}{dt}(i,t)  \geq \epsilon.$$

This shows that the minimum of $\tilde z$ is also achieved at $T$, hence for all $(j,m,t)$
$$z_m(j,t)+\epsilon(T-t)=\tilde z_m(j,t)\geq \tilde z_n(i,T)=z_n(i,T).$$
Letting $\epsilon\to 0$, we get
$$z_m(j,t)\geq \min_{n,i} z_n(i,T).    $$
Hence
$$z_m(j,t)\geq  - \|z(T)\|_{\infty}.  $$
Therefore we have $\|z(t)\|_{\infty} \leq \|z(T)\|_{\infty}$.
\end{proof}

Note: let $z(t)=K(t,s)z(s)$ be a solution of \eqref{unpert2} with terminal data $z(s)=b$, then  lemma \ref{contprop} implies that $\|z(t)\|_{\infty} \leq \|z(s)\|_{\infty}$, and therefore
 \begin{equation}\label{Kineq}\|K(t,s)b\|_{\infty} \leq \|b\|_{\infty}, \forall \,b .\end{equation}

From the previous lemma we also conclude
\begin{lemma}
If $p_1\leq p_2$, and $t\leq s$, then we have
\[
K(t,s)p_1\leq K(t, s)p_2.
\]
\end{lemma}
\begin{proof}
Observe that if $p_1-p_2\leq 0$ then $K(t, s)(p_1-p_2)\leq 0$, by lemma \ref{contprop}.
\end{proof}

We note now that if $t\leq s\leq T$ we have $K(t,s)K(s,T)=K(t,T)$, which implies
\[
\frac{d}{ds}\bigg( K(t,s)K(s, T)\bigg)=0.
\]
Hence, using equation \eqref{unpert4} we get
\[
-K(t,s)M(s)K(s,T)+\left(\frac{d}{ds} K(t,s)\right)  K(s,T)=0,
\]
and therefore, by taking $T=s$ we conclude that
\begin{equation}
\label{uidt}
 \frac{d}{ds} K(t,s)=K(t, s)M(s).
\end{equation}

We now prove the main technical lemma:
\begin{lemma}\label{413}
Suppose $z$ is a solution to
\begin{equation}
\label{difeneq}
-\dot z(s)\leq M(s) z(s) +f(z(s)).
\end{equation}
%where $f(z)\leq C\|z\|_{\infty}+CN\|z\|_{\infty}^2$.
Then
\[
z(t)\leq \|z(T)\|_{\infty}+\int_t^T \|f(z(s))\|_{\infty}ds.
\]
\end{lemma}
\begin{proof}
Multiplying \eqref{difeneq} by the order preserving operator $K(t,s)$, we have
$$ -K(t,s)\dot z(s)\leq K(t,s)M(s) z(s) +K(t,s)f(z(s))$$
using the identity
\[
\frac{d}{ds} K(t,s)z(s)=K(t,s) \dot z(s)+ K(t, s) M(s) z(s),
\]
which follows from \eqref{uidt},
we get
\[
-\frac{d}{ds} \Big(K(t,s)z(s)\Big) + K(t, s) M(s) z(s)\leq K(t,s)M(s) z(s)+ K(t, s)f(z(s)).
\]
%
%\textbf{PERGUNTA $M(s) K(s, t)= K(s,t)M(s) $ ?}
%
%\bigskip
%
%
Thus, integrating between $t$ and $T$, we have
\[
z(t)-K(t, T)z(T)\leq \int_t^T K(t, s)f(z(s)) ds.
\]
So, using equation \eqref{Kineq},
\[
z(t)\leq  \|z(T)\|_{\infty}+\int_t^T \|f(z(s))\|_{\infty}ds.
\]
\end{proof}

%QUESTAO: deveríamos nos certificar de que a função $s \rightarrow  z_{n_s}(i_s,s)$ é mesmo diferenciavel.
%Podem ocorrer problemas nos pontos de descontinuidade de $s \rightarrow (i_s,n_s)$ ?
%
%NOTA (DG) Isto não tem problema porque não estamos a avaliar nenhuma função em $(i_s, n_:s)$, mas sim em $(i, n)$ fixos.

%\newpage
\subsection{Uniform estimates}\label{uest}

In this section we prove "gradient estimates" for the $N+1$ player game, that is,
we assume that the difference $u_{n+1}-u_n$ is of the order $\frac 1 N$ at time $T$
and show that it remains so for $0\leq t\leq T$, as long as $T$ is sufficiently small.

We start by establishing an auxiliary result:
\begin{lemma}\label{tfc}
Suppose $v=v(s)$ is a solution to the ODE with terminal condition
\begin{equation}\label{tfc_p}
\begin{cases}
-\frac{dv}{ds}= Cv+C N v^2 + \frac{C}{N}\\
v(T)\leq\frac{C}{N},
\end{cases}
\end{equation}
where $N$ is a natural number, and $C>0$. Then, there exists $T^{\star}>0$, which does not depend on $N$, such that
$T \leq T^{\star}$ implies
%$\frac{C}{N} \leq
$v(s) \leq \frac{2C}{N}$ for all $0\leq s \leq T$.
\end{lemma}
\begin{proof}
Note that (\ref{tfc_p}) implies that
  $v$ is a monotone decreasing function of $s$ and
 is equivalent to
%\begin{equation}\label{tfc_s}
$$\begin{cases}
\frac{ds}{d v}= \frac{-1}{Cv+C N v^2 + \frac{C}{N}}\\
s(\frac{C}{N})\leq T.
\end{cases}
$$%\end{equation}
This implies by direct integration that
$$ s\left(\frac{2C}{N}\right)
%s\left(\frac{C}{N}\right)- \int_{\frac{C}{N}}^{\frac{2C}{N}} \frac{dv}{Cv+C N v^2 + \frac{C}{N}}
\leq T - \int_{\frac{C}{N}}^{\frac{2C}{N}} \frac{dv}{Cv+C N v^2 + \frac{C}{N}}\,. $$ Now
$$\int_{\frac{C}{N}}^{\frac{2C}{N}} \frac{dv}{Cv+C N v^2 + \frac{C}{N}} \geq
\int_{\frac{C}{N}}^{\frac{2C}{N}} \frac{N}{2C^2+4C^3+C}dv = \frac{1}{2C+4C^2+1}. $$
Therefore if we define $T^{\star}=\frac{1}{2C+4C^2+1} $, we  have that $s\left(\frac{2C}N\right) \leq 0$ if $T \leq T^{\star}$.
%Considering that  $v$ is a monotone decreasing function of $s$,
Hence this implies $v(0) \leq \frac{2C}N$, which yields the desired result when we take into account that $v$ is a decreasing function of $s$.
% and also that $v(s)\leq v(0)\leq 2C/N$ for all $0\leq s \leq T$ .
\end{proof}

\begin{proposition}
\label{lipbounds}
Suppose that
\begin{equation}\label{lbh}
\|u_{n+1}(T)-u_n(T)\|_{\infty}\leq \frac C N.
\end{equation}
for $C>0$.
Let $u$ be a solution of \eqref{eqhj}. Then there exists $T^{\star}>0$ such that, for   $0<T< T^{\star}$ ,we have
\[
\|u_{n+1}(t)-u_n(t)\|_{\infty}\leq \frac {2C} N,
\] for all $0\leq t\leq T$.
\end{proposition}
\begin{proof}

Let
\[
z_n=u_{n+1}-u_n.
\]
Note that, as usual, $z_n=(z_n^0, z_n^1)$.
We have
\[
-\dot z_n=\gamma^+_{n+1}z_{n+1}-\gamma^+_n z_n-\gamma^-_{n+1}z_{n}+\gamma^-_{n}z_{n-1}+
h\left(\frac {n+1} N, i, \bar u_{n+1}-u_{n+1}\right)-h\left(\frac n N, i, \bar u_{n}-u_{n}\right).
\]
We can write
\begin{align*}
\gamma^+_{n+1}z_{n+1}-&\gamma^+_n z_n-\gamma^-_{n+1}z_{n}+\gamma^-_{n}z_{n-1}\\
=&\frac{\gamma^+_{n+1}+\gamma^+_n}{2} (z_{n+1}-z_n)+\frac{\gamma^+_{n+1}-\gamma^+_n}{2} (z_{n+1}+z_n)\\
&+\frac{\gamma^-_{n+1}+\gamma^-_n}{2} (z_{n-1}-z_n)+\frac{\gamma^-_{n}-\gamma^-_{n+1}}{2} (z_{n-1}+z_n).
\end{align*}
We must now observe that
\[
\left|\frac{\gamma^-_{n}-\gamma^-_{n+1}}{2}\right|\leq C+CN\|z\|_{\infty},
\]
as well as
\[
\left|\frac{\gamma^+_{n+1}-\gamma^+_n}{2}\right|\leq C+CN\|z\|_{\infty}.
\]
Furthermore, we have
\begin{align*}
&h\left(\frac {n+1} N, i, \bar u_{n+1}-u_{n+1}\right)-h\left(\frac n N, i, \bar u_{n}-u_{n}\right)\\
&=h\left(\frac {n+1} N, i, \bar u_{n+1}-u_{n+1}\right)-h\left(\frac {n} N, i, \bar u_{n+1}-u_{n+1}\right)\\
&+h\left(\frac {n} N, i, \bar u_{n+1}-u_{n+1}\right)-h\left(\frac n N, i, \bar u_{n}-u_{n}\right)\\
&\leq \frac{C} N+h_p\left(\frac n N, i, \bar u_{n}-u_{n}\right) \left(( \bar u_{n+1}-u_{n+1})-(\bar u_{n}-u_{n})\right)\\
&\leq \frac{C}N+\mu_n (\bar z_n-z_n),
\end{align*}
where $\mu_n=h_p\left(\frac n N, i, \bar u_{n}-u_{n}\right)\geq 0$.

At this point we are in position to apply lemma \ref{413} from the previous section. We obtain
$$z_n(t)=(u_{n+1}-u_{n})(t)\leq \|z(T)\|_{\infty} +\int_t^T C\|z(s)\|_{\infty}+C\|z(s)\|_{\infty}^2+\frac {C} N  \;ds\;. $$
We can also use the same
argument applied to
\[
\tilde z_n=u_{n}-u_{n+1}\,.
\]
Finally, if we set $w=\|u_{n+1}-u_n\|_{\infty}$ we conclude that
\[
w(t)\leq w(T)+\int_t^T Cw(s)+CN w(s)^2+\frac{C}{N} ds.
\]
Now we define
$$\eta(t) = w(T)+\int_t^T Cw(s)+CN w(s)^2+\frac{C}{N} ds.$$
We have that
\begin{equation}\label{blabla}
w(t)\leq \eta(t),
\end{equation}
 and also that $$\frac{d \eta}{dt}(t) = - g(w(t)),$$ where $g$ is the nondecreasing function $g(w)= Cw+CN w^2+\frac{C}{N}$. Thus
$$
\begin{cases}
\frac{d \eta}{dt}(t) \geq - g(\eta(t)) \\
\eta(T)=w(T).
\end{cases}
$$
%Then a variant of the Gronwall inequality shows that if $T$ is small enough and $w(T)\leq \frac C N$ we have that $w(t)$ is uniformly bounded for all $0\leq t\leq T$, uniformly in $N$.
A standard argument from the basic theory of differential inequalities can now be used to prove that
$\eta(t) \leq v(t)$ for $0 \leq t \leq T$ if $v(t)$
is the solution of
$$
\begin{cases}
\frac{d v}{dt}(t) = - g(v(t)) \\
v(T)=w(T).
\end{cases}
$$
This last result can be combined with lemma \ref{tfc}, the hypothesis $w(T)\leq \frac{C}N$ %(\ref{lbh})
 and the inequality (\ref{blabla}), to
prove that $w(t) \leq \frac{2C}N$ for all $0 \leq t \leq T$, which ends the proof of the proposition.
\end{proof}

\subsection{Convergence}
\label{convsubsec}

In this section we  prove theorem \ref{teoconv}, which implies the convergence of both distribution and value function of the $N+1$-player game to the mean field game, for small times.

We start by assuming that  at the initial time the $N$ players distinct from the reference player distribute themselves between states $0$ and $1$ according
to a Bernoulli distribution with probability $\bar \theta$ of being in state $0$.

Let
\begin{equation} \label{VNQN}
\begin{cases}
V_N(t) \equiv \mathbb{E}\left[ \left( \frac{n(t)}{N} - \theta(t) \right)^2  \right]\,, \\  \\
 W_N(t) \equiv \mathbb{E}\left[ \left( u(0,t) - u_{n(t)}(0,t) \right)^2  \right]\,,  \\ \\
 \bar W_N(t) \equiv \mathbb{E}\left[ \left( u(1,t) - u_{n(t)}(1,t) \right)^2  \right]\,, \\  \\
 Q_N(t) \equiv W_N(t)+\bar W_N(t)\,,
\end{cases}
\end{equation}
where $\theta(t)$ is the solution of  (\ref{eq_pi}),  $0 \leq n(t)\leq N$ is the number of players
(distinct from the reference player)
which are in state $0$ at time $t$, and  $u=u(i,t)$ and $u_n=u_n(i,t)$ are respectively the solution of the HJ equation and terminal conditions for the MFG \eqref{MFG_eq} and $N+1$ player game \eqref{eqhj}.

We have $$V_N(0)= \mbox{Var}\left[ \frac{n(0)}{N}\right] = \frac{\bar \theta(1-\bar \theta)}{N}\,,$$ because $n(0)$ is the sum of $N$ iid rv with Bernoulli distribution.
%\footnote{precisamos colocar algum comentario a mais sobre isto, provavelmente no inicio do texto.}
%Note that here we deal with expectations and no conditional expectations. \footnote{vale a pena deixar este comentário?}

In this section $\alpha = \alpha(i,t)$ is the optimal control  for the MFG, while $\alpha^N=\alpha^N(i,n,t)$ is the optimal control for the $N+1$ player game.
We know from sections \ref{hypot},  \ref{hjode} and \ref{tcpmfm}  that
$\alpha^N= \alpha^*\left(\bar u_n- u_n,\frac{n}N,i\right)$
and $\alpha= \alpha^*(\bar u- u, \theta ,i)$.

\begin{lemma}\label{conv1}
There exists $C_1>0$ such that
$$V_N(t) \leq \int_0^t C_1 (V_N(s)+ Q_N(s)) ds + \frac{C_1}N.$$
\end{lemma}
\begin{proof}
Using Dynkin's Formula (\ref{Dynkin}) with $\varphi(i,n,s)=\left(\frac{n(s)}N-\theta(s)\right)^2$, we have
$$ V_N(t) -  \frac{\theta_0(1-\theta_0)}{N} =
 \mathbb{E} \int_0^t \omega_N(s) + \varsigma_N(s)
ds$$ where
$$
\omega_N(s) = (N-n)\alpha^N_1 \left[ \left( \frac{n+1}{N}-\theta \right)^2 - \left( \frac{n}{N}-\theta \right)^2 \right]+
n \alpha^N_0 \left[ \left( \frac{n-1}{N}-\theta \right)^2
- \left( \frac{n}{N}-\theta \right)^2 \right]\,,$$
$$
\alpha_0^N=\alpha^*\left(\bar u_{n-i}-u_{n-i},\frac{n-i}N,0\right)\,,
$$
$$
\alpha_1^N=\alpha^*\left(\bar u_{n+1-i}-u_{n+1-i},\frac{n+1-i}N,1\right)\,,
$$
$$u_n=u_N(i,n,t)\,,$$
and $$
 \varsigma_N(s) = \frac{d \varphi}{dt}(i,n,r)=
 - 2  \left( \frac{n}{N}-\theta \right)
\left((1-\theta) \alpha_1 - \theta \alpha_0 \right).
$$
%(using $\theta^{\prime}=(1-2\theta) \alpha^N$,)
We have
\begin{align*}
\omega_N(s) =&  \left(1-\frac{n}{N}\right)\alpha^N_1 \left( \frac{2n+1}{N}-2\theta
 \right)-
\frac{n}{N} \alpha^N_0 \left(  \frac{2n-1}{N}-2 \theta
 \right)\\
%$$=  2\alpha^N_1 \left(1-\frac{n}{N}\right) \left( \frac{n}{N}-\theta  \right)- 2 \alpha^N_0
%\frac{n}{N}  \left(  \frac{n}{N}- \theta  \right) + \frac{\alpha^N_1}{N} +\frac{n}{N^2}(\alpha^N_0-\alpha^N_1)= $$
= &
 2\alpha^N_1 \left(1-\frac{n}{N}\right) \left( \frac{n}{N}-\theta
 \right)- 2 \alpha^N_0
\frac{n}{N}  \left(  \frac{n}{N}- \theta  \right) + \tau_N(s),
\end{align*}
where $ \tau_N(s) = \frac{\alpha^N_1}{N} +\frac{n}{N^2}(\alpha^N_0-\alpha^N_1)\,.$
Now
\begin{align*}
\omega_N(s) +& \varsigma_N(s)
= 2 \left( \frac{n}{N}-\theta \right)
\left[ \alpha^N_1 \left(1-\frac{n}{N}\right) - \alpha^N_0 \frac{n}{N}
-\left((1-\theta) \alpha_1 - \theta \alpha_0 \right) \right] + \tau_N(s)
\\
=&  2 \left( \frac{n}{N}-\theta \right)
\left[ (\alpha^N_1+\alpha^N_0)  \left(-\frac{n}{N}\right) +(\alpha_1+\alpha_0)\theta + (\alpha^N_1-\alpha_1)
\right] + \tau_N(s)\\
= & 2 \left( \frac{n}{N}-\theta \right)
\left[ (\alpha^N_1+\alpha^N_0)  \left(\theta-\frac{n}{N}\right) +(\alpha_1-\alpha^N_1+\alpha_0-\alpha^N_0)\theta + (\alpha^N_1-\alpha_1)
\right] + \tau_N(s)
\\=& - 2  (\alpha^N_0+\alpha^N_1) \left( \frac{n}{N} - \theta \right)^2 + 2 \left( \frac{n}{N}-\theta \right) \left((\alpha_1-\alpha^N_1+\alpha_0-\alpha^N_0)\theta + (\alpha^N_1-\alpha_1)\right)
+  \tau_N(s).
\end{align*}

Then
\begin{align*}
V_N(t) -  \frac{\theta_0(1-\theta_0)}{N} =&
 - 2 \mathbb{E}\left[ \int_0^t  (\alpha^N_0+\alpha^N_1) \left( \frac{n}{N} - \theta \right)^2 ds \right] \\
 &+
 \mathbb{E} \left[ \int_0^t 2 \left( \frac{n}{N}-\theta \right) \left( (\alpha_1-\alpha^N_1+\alpha_0-\alpha^N_0)\theta + (\alpha^N_1-\alpha_1) \right) ds
 \right]\\
 &+
  \mathbb{E} \left[ \int_0^t \tau_N(s) ds \right].
\end{align*}

Now we see that
\begin{align*}|\alpha_0-\alpha^N_0|&=\left|\alpha^*(\bar u-u,\theta,0)-
\alpha^*\left(\bar u_{n-i}-u_{n-i},\frac{n-i}N,0\right)\right| \\
&< K\bigg( \left|\theta-\frac{n-i}N\right|+|\bar u-\bar u_{n-i}| + |u - u_{n-i}|\bigg)\\
& <  K\bigg( \left|\theta-\frac{n}N\right|+|\bar u-\bar u_n| + |\bar u_{n-i}-\bar u_n| + |u - u_n|+ |u_{n-i} - u_n|+\frac{1}N\bigg)\\
& <  K\bigg( \left|\theta-\frac{n}N\right|+|\bar u-\bar u_n| +  |u - u_n|+ \frac 3 N\bigg)\,,
\end{align*}
where we used that $\alpha^*$ is Lipschitz in both variables,
and $u$ and $u^N$ are bounded, and the uniform bounds on $|u_{n+1}-u_n|$ obtained in proposition \ref{lipbounds}
of \S \ref{uest}.
Similarly
$$
|\alpha_1-\alpha^N_1| <  K\left( \left|\theta-\frac{n}{N}\right|+|\bar u-\bar u_n| + |u - u_n|+\frac 3 N\right).
$$

Thus
\begin{align*}
V_N(t) \leq & K_1 \int_0^t V_N(s) ds+2 \mathbb{E} \int_0^t \bigg(\frac{n}{N}-\theta\bigg)
K\bigg( \left|\theta-\frac{n}{N}\right|+|\bar u-\bar u_n| + |u - u_n|+\frac 3 N\bigg)ds+\frac{K_2}N\\
\leq &
(K_1 +2K)\int_0^t V_N(s) ds+2 \mathbb{E} \int_0^t \bigg(\frac{n}{N}-\theta\bigg)
K\bigg( \left|\bar u-\bar u_n\right| + |u - u_n|+\frac 3 N\bigg)ds+\frac{K_2}N\\
\leq &
(K_1 +2K)\int_0^t V_N(s) ds+2 K \int_0^t 2V_N(s)+
(W_N(s) + \bar W_N(s) )ds+\frac{K_2+6T}N\\
=&
\int_0^t K_3 V_N(s) + 2K
Q_N(s)ds+\frac{K_2+6T}N\\
\leq & \int_0^t C_1 (V_N(s)+ Q_N(s)) ds + \frac{C_1}N\,.
\end{align*}
\end{proof}

\begin{lemma}\label{conv2} There exists $C_2>0$ such that
%If $ Q_N(t) = W_N(t)+ \bar W_N(t)$, then
$$Q_N(t) \leq \int_t^T C_2 (V_N(s)+ Q_N(s)) ds + \frac{C_2}N. $$
\end{lemma}
\begin{proof}
In this proof, $u_n(s)$ or simply $u_n$ will denote the expected minimum cost of player $N+1$ conditioned on its state being equal to $0$ at time $s$, i.e., $u_{n(s)}(0, s)$. We will also use, here, $u(s)$ or simply $u$ to denote $u(0,s)$.

Using Dynkin formula
%(\ref{Dynkin}) \footnote{DF está enunciada como uma esperança condicional. Seria interessante colocar algum comentário sobre %isto, aqui ou no início do texto?}
(\ref{Dynkin}) with $\varphi(i,n,s)=\big(u_{n(s)}\left(0,s\right)-u(0,s)\big)^2$, and equations
(\ref{eqhj}) and (\ref{HJ_MFG}),
 we have
%\footnote{explicar exatamente por que nao entra a parte em $\bar u - u$: pois fazemos o controle de n fixo...}
\begin{align*}
&W_N(t)-W_N(T)=-\mathbb{E}[(u_n(t)-u(t))^2]+\mathbb{E}[(u_n(T)-u(T))^2]\\
=& \mathbb{E} \int_t^T 2 (u_n-u) \frac{d}{ds} (u_n-u) ds \\
&\qquad +\mathbb{E} \int_t^T \gamma_n^+ \left[ (u_{n+1}-u)^2 - (u_n-u)^2 \right]
+ \gamma_n^- \left[ (u_{n-1}-u)^2 - (u_n-u)^2 \right]
ds \\
=& \mathbb{E} \int_t^T 2 (u_n-u) \left(-\gamma_n^+(u_{n+1}-u_n) - \gamma_n^-(u_{n-1}-u_n)
-h\bigg(\bar u_n-u_n,\frac{n}{N},0\bigg)+h(\bar u-u,\theta,0)\right) ds\\
&\qquad + \mathbb{E} \int_t^T\gamma_n^+ \left[ (u_{n+1}-u)^2 - (u_n-u)^2 \right]
+ \gamma_n^- \left[ (u_{n-1}-u)^2 - (u_n-u)^2 \right]
ds\\
=&\mathbb{E}\int_t^T \gamma_n^+ (u_{n+1}-u_n)^2 + \gamma_n^- (u_{n-1}-u_n)^2 - 2 \left( h\left(\bar u_n-u_n,\frac{n}{N},0\right)-h(\bar u-u,\theta,0)\right)(u_n-u)ds,
\end{align*}
where $\gamma_n^{\pm}=\gamma_n^{\pm}(0,n(s),s)$.
In the last equation we used the fact that
$$- 2 (u_n-u)  \gamma_n^+(u_{n+1}-u_n) + \gamma_n^+ \left[ (u_{n+1}-u)^2 - (u_n-u)^2 \right]
=
%$$
%$$= \gamma_n^+ \left[ -(2u_n u_{n+1}
%-2u_n^2-2u u_{n+1} + 2u u_n)+
%( u_{n+1}^2 - 2 uu_{n+1} - u_n^2+2 u u_n  )
% \right]=$$
%$$=
\gamma_n^+ \left( u_{n+1}-u_n \right)^2\,,$$
and a similar calculation  for $\gamma_n^-$.

Now, using results from \S \ref{uest}, proposition \ref{lipbounds}, we have that
$\gamma_n^+(u_{n+1}-u_n)^2$,
 $\gamma_n^-(u_{n-1}-u_n)^2$ and $W_N(T)$ are bounded by $\frac{K_5}N$,
which implies
$$W_N(t) \leq \frac{K_6}N +
2 \mathbb{E}\int_t^T \left( h\left(\bar u_n-u_n,\frac{n}{N},0\right)-h\left(\bar u-u,\theta,0\right)\right)(u_n-u)ds\,.$$
Using the fact that $h$ is Lipschitz in both variables,
we have
$$\left|h\left(\bar u_n-u_n,\frac{n}{N},0\right)-h(\bar u-u,\theta,0)\right|<
 K\bigg(\left |\theta-\frac{n}{N}\right|+|\bar u-\bar u_n| + |u - u_n|\bigg)\,.$$
Thus
$$ W_N(t) \leq \frac{K_6}N  + K_7 \int_t^T
 V_N(s) +
W_N(s) + \bar W_N(s) ds\,.$$
With a similar calculation we have a analogous inequality for $\bar W_N(t)$, which ends the proof.
\end{proof}

Now we can state and prove our main result that establishes the convergence
of the $N+1$ player game to the mean field model as $N\to \infty$.

\begin{theorem}\label{teoconv} If $\rho = TC <1$, where $C= \max\{C_1,C_2\}$,  and $Q_N(t)+V_N(t)$ is given in (\ref{VNQN})
then
$$Q_N(t)+V_N(t) \leq \frac{C}{1-\rho} \frac{1}{N}\,\;\;\forall t \in [0,T]\,.$$
\end{theorem}
\begin{proof}
Adding both inequalities given in the two last lemmas, we have
$$Q_N(t)+V_N(t) \leq C \int_0^T (V_N(s)+Q_N(s))ds + \frac C N\,.$$
Now suppose $\rho = TC <1$.
Defining
\[
Q_N+V_N = \max_{0 \leq t \leq T}Q_N(t)+V_N(t),
\]
we have
$$Q_N+V_N \leq \rho (Q_N+V_N) + \frac C N, $$
which proves the theorem.
\end{proof}

%\\section{Final Remarks}

%\subsection{Possible extensions}

%\subsection{Conjectures and additional results}

%\subsection{Applications}

\bibliographystyle{alpha}

\bibliography{mfg}

\end{document}